\documentclass[11pt,twoside, a4paper, english, reqno]{amsart}
\usepackage{graphicx, amsmath, varioref, amscd, amssymb,color, bm, stmaryrd, amsthm, epsfig, color, epic}
\usepackage{fancybox}
\usepackage[pdftex, colorlinks=true]{hyperref}
\usepackage{upref}
\usepackage{pdfsync}

\usepackage{eso-pic}
\usepackage{color}
\usepackage{type1cm}
\makeatletter
  \AddToShipoutPicture*{%
    \setlength{\@tempdimb}{.12\paperwidth}%
    \setlength{\@tempdimc}{.5\paperheight}%
    \setlength{\unitlength}{1pt}%
    \put(\strip@pt\@tempdimb,\strip@pt\@tempdimc){%
      \makebox(0,0){\rotatebox{90}{\textcolor[gray]{0.75}{\fontsize{2cm}{2cm}}}}
    }
}

\makeatother

\numberwithin{equation}{section}
\allowdisplaybreaks

\newtheorem{theorem}{Theorem}[section]
\newtheorem{lemma}[theorem]{Lemma}

\newtheorem{proposition}[theorem]{Proposition}
\theoremstyle{definition}
\newtheorem{definition}[theorem]{Definition}

\theoremstyle{remark}

\newcommand{\Div}{\operatorname{div}}

\newcommand{\Grad}{\nabla}
\newcommand{\vr}{\varrho}

\newcommand{\dx}{{\rm d}x}

\newcommand{\vc}[1]{{\bm{#1}}}

\newcommand{\R}{\mathbb{R}}

\newcommand{\e}{\varepsilon}

\begin{document}

\title[On a tumor growth model] {On a nonlinear model for tumor growth with drug application}

\author[Donatelli]{Donatella Donatelli}
\address[Donatelli]{\newline
Departement of Engineering Computer Science and Mathematics\\
University of L'Aquila\\
67100 L'Aquila, Italy.}
\email[]{\href{donatell@univaq.it}{donatell@univaq.it}}
\urladdr{\href{http://univaq.it/~donatell}{univaq.it/\~{}donatell}}

\author[Trivisa]{Konstantina Trivisa}
\address[Trivisa]{\newline
Department of Mathematics \\ University of Maryland \\ College Park, MD 20742-4015, USA.}
\email[]{\href{http://www.math.umd.edu}{trivisa@math.umd.edu}}
\urladdr{\href{http://www.math.umd.edu/~trivisa}{math.umd.edu/\~{}trivisa}}

\date{\today}

\subjclass[2010]{Primary: 35Q30, 76N10; Secondary: 46E35.}

\keywords{Tumor growth models, cancer progression, mixed models, moving domain, penalization, existence.}

\thanks{}

\maketitle

\begin{abstract}
We investigate the dynamics of a nonlinear system modeling tumor growth with drug application. The tumor is viewed as a mixture consisting of proliferating, quiescent and dead cells as well as a nutrient  in the presence of a drug. 
The system is given by a multi-phase flow model: the densities of the different cells are governed by a set of  transport equations, the density of the nutrient and the density of the drug are governed by rather general diffusion equations, while the velocity of the tumor is given by Brinkman's equation. The domain occupied by the tumor in this setting is a growing continuum  $\Omega$ with boundary $\partial \Omega$  both of which  evolve in time. Global-in-time weak solutions are obtained using an  approach based on penalization of the boundary behavior, diffusion and  viscosity in the weak formulation. Both the solutions and the domain are rather general, no symmetry assumption is required and the result holds for large initial data. This article is part of a research program whose aim is the investigation of the effect of drug application  in tumor growth.
\end{abstract}

\tableofcontents{}

\section{Introduction}\label{S1}

\subsection{Motivation}
The investigation of the effect of drug application in the treatment of cancer is the subject of intense scientific effort. A major cause of the failure of chemotherapeutic treatments for cancer is the development of resistance to drugs. This article is part of a research program whose aim is the investigation of the effect of drug application on tumor growth.
We investigate the dynamics of a nonlinear system describing the evolution of cancerous cells.
In this setting, the tumor is viewed as a mixture consisting of proliferating, quiescent and dead cells in the presence of a nutrient (oxygen) and drug.
The mathematical model presented here  is governed by 
\begin{itemize}
\item a system of transport equations, which describe the evolution of the densities of the cells that are present in the tumor: proliferating cells with density $P$, quiescent cells with density $Q$ and dead cells with density $D$ (this part of the tumor includes what is known also as waste or extra-cellular medium),

\item two rather general diffusion equations which are used to describe the diffusion  of the nutrient (oxygen) within the tumor region and  the evolution of the drug within the same regime.  In general, these equations obey Fick's law: the nutrient is consumed at a rate proportional to the rate of cell mitosis, whereas the drug is consumed at a rate which is determined by the drug effectiveness,  

\item an extension of the Darcy law, known as Brinkman's equation, which determines the velocity field. The  continuous movement within the tumor region is  due to proliferation, mitosis, apoptosis or removal of cells. Note, the tumor in the present context is viewed as a fluid-like porous medium.
\end{itemize}


Motivated by the experiment of Roda {\em et al.} (2011, 2012) and the mathematical analysis in Friedman {\em et al.} \cite{Friedman-2004}, \cite{ChenFriedman-2013},  and Zhao in \cite{Zhao-2010} our model is based on  the following biological principles: 
\begin{enumerate}
\item[{\bf [P1]}] Living cells are either in a   {\em proliferating phase} or in a  {\em quiescent phase}.

\item[{\bf [P2]}] Proliferating cells die as a result of {\em apoptosis,} which is a cell-loss mechanism. Quiescent cells die in part due to {\em apoptosis} and more often  due to starvation. In fact the proliferation and the necrotic death rates of tumor cells depend on the oxygen level.
\item [{\bf [P3]}] The dead tumor cells are obtained from necrosis and apoptosis of live tumor cells, and they are cleared by macrophages.

\item [{\bf [P4]}] Living cells undergo {\em mitosis}, a process that takes place in the nucleus of a dividing cell. 
\item[{\bf [P5]}] Cells change from quiescent phase into proliferating phase at a rate which increases with the nutrient level, and they die at a rate which increases as the level of nutrient (oxygen)  decreases.
\item[{\bf [P6]}]  Proliferating cells become quiescent   and die at a rate which increases as the nutrient concentration decreases. The proliferation rate increases with the nutrient concentration.
\item[{\bf [P7]}] Proliferating cells and quiescent cells become dead cells at a rate which depends on the drug concentration.
\end{enumerate}
  
The tumor region $\Omega_t :=\Omega(t)$ is contained in a fixed domain $B$ and the region 
$B\setminus \Omega_t$ represents the healthy tissue (see \figurename~\ref{regions}).  The tumor region 
$\Omega_t$ and its boundary $\partial\Omega_t$ evolve with respect to time. Both live and dead tumor cells are assumed to be in the tumor region $\Omega_t;$ oxygen molecules can diffuse throughout the whole domain $B$.  Abnormal proliferation of tumor cells generates internal pressure in $\Omega(t)$, resulting to a velocity field $\vc{v} \not= 0$ (while $\vc{v} = 0$ in $B\setminus \Omega_t$).


\begin{figure}[htbp] 
\centering
\includegraphics[width=5cm]{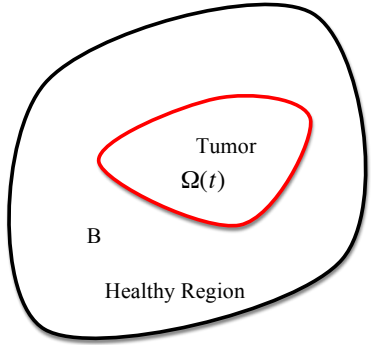}
\caption{Healthy tissue - Tumor regime.} 
\label{regions} 
\end{figure}

\subsection{Governing equations of cells, oxygen and drug}
\subsubsection{Transport equations for the evolution of the cell densities}
All the cells are assumed to follow the general continuity equation:
\begin{equation}
\frac{\partial \vr}{\partial t} + \Grad \cdot (\vr  {\vc{v}}) = G, \nonumber
\end{equation}
where $\vr$ may represent  densities of proliferating/quiescent and dead cells. The function $G$ includes in general proliferation, apoptosis or clearance of cells, and chemotaxis terms as appropriate.

Due to proliferation and removal of cells, there is a continuous motion within the tumor represented by a velocity field ${\vc{v}}$. 
We assume that there are three types of cells: proliferative cells with density $P,$ quiescent cells with density $Q$ and dead cells with density  $D$ in the presence of a nutrient (oxygen) with density $C$
and a drug with density $W.$
The rates of change from one phase to another are functions of the nutrient concentration C:
$$ P \to Q \,\, \mbox{at rate}\,\, K_Q (C),$$
$$ Q \to P \,\, \mbox{at rate}\,\, K_P (C), $$
$$ P \to D \,\, \mbox{at rate}\,\, K_A  (C), $$
$$ Q \to D \,\, \mbox{at rate}\,\, K_D (C), $$
where $K_A$ stands for apoptosis. Finally, dead cells are removed at rate $K_R$ (independent of $C$), and the rate of cell proliferation (new births) is $K_B©.$ 
\smallskip

\subsubsection{The tumor tissue as a porous medium}
Due to proliferation and removal of cells there is continuous motion of cells within the tumor; this movement is represented by the velocity field $\vc{v}$ given by an alternative to Darcy's equation known as {\em Brinkman's equation}
\begin{equation}
\Grad \sigma = - \frac{\mu}{K} \vc{v} + \mu \Delta \vc{v}.\label{pressure2}
\end{equation}

\noindent
where $\sigma$ denotes the pressure, $\mu$  is a positive constant describing the viscous like properties of tumor cells, whereas $K$ denotes the permeability. 

Relation \eqref{pressure2} includes two viscous terms. The first term is the usual Darcy law and the second is analogous to the Laplacian term that appears in the Navier-Stokes equation. 
At a first look, \eqref{pressure2} appears as an over damped force balance. A second interpretation of  this relation states that the tumor tissue is ``fluid like" and that the tumor cells flow through the fixed extracellular matrix like a flow through a porous medium, obeying Brinkman's law.

The mass conservation laws for the densities of the proliferative cells $P,$ quiescent cells $Q$ and dead cells $D$  in $\Omega(t)$ 
take the following form:
\begin{equation}
 \frac{\partial P}{\partial t} + \Div (P \vc{v}) = \vc{G_P}, \label{dP}
 \end{equation}
 \begin{equation}
 \frac{\partial Q}{\partial t} + \Div (Q \vc{v}) = \vc{G_Q}, \label{dQ}
\end{equation}
\begin{equation}
\frac{\partial D}{\partial t} + \Div (D \vc{v}) = \vc{G_D}. \label{dD}
\end{equation}
Following Friedman \cite{Friedman-2004}, the source terms ${\bf \{G_P, G_Q, G_D\}}$ are of the following form:
\begin{equation}\label{Gp}
 \vc{G_P} =  \left(K_B C - K_Q (\bar C- C) - K_A(\bar C - C)\right) P + K_P C Q - i_1 G_1(W) P, 
\end{equation}
where  $G_1(\cdot)$  a  smooth function and $K_{B}$, $K_{Q}$, $K_{A}$ are positive constants. The first term in this equation accounts for the increase of the number of cells due to new births, loss due to change of phase from proliferating to quiescent and loss due to apoptosis. The second term reflects the increase of the number of proliferating cells generated from quiescent cells, whereas the third term accounts for the decrease of the number of cells due to death resulting from the  effect of drug. In an analogous fashion
\begin{equation}\label{Gq}
 \vc{G_Q} = K_Q (\bar C- C)P - \left(K_P C + K_D(\bar C - C)\right) Q - i_2 G_2(W) Q,
\end{equation}
with $G_2(\cdot)$ a  smooth function and $K_{P}$, $K_{Q}$, $K_{D}$  positive constants. In the above relations \eqref{Gp}-\eqref{Gq} $i_1 G_1(W)$ and   $i_2 G_2(W)$ denote the rates by which  the proliferating cells and the quiescent  cells  become dead cells due to the drug. Finally,
\begin{equation}\label{Gd}
\vc{G_D} =  K_A (\bar C- C)P  + K_D (\bar C-C) Q - K_R D + i_1G_1(W)P + i_2 G_2(W)Q.
\end{equation}

\smallskip
\subsubsection{A linear diffusion equation for the evolution of nutrient}

Tumor cells consume nutrients (oxygen). In contrast to the equations of cell densities, the equations of the oxygen molecules in the tumor include diffusion terms in the following form:

\begin{equation}
\frac{\partial C}{\partial t} = \Grad \cdot (\nu_1 \Grad C) -  \left(K_1 K_P CP + K_2 K_Q(\bar{C}-C)Q\right)C. \nonumber
\end{equation}
Assuming that $\nu_1$ is constant this equation (cf. Friedman \cite{Friedman-2004}) becomes 
\begin{equation}\label{dC}
\frac{\partial C}{\partial t} = \nu_1 \Delta C -  \left(K_1 K_P CP + K_2 K_Q(\bar{C}-C)Q\right)C.
\end{equation}
This equation describes the diffusion of the oxygen in the tumor region. According to (cf. Ward and King \cite{WardKing-1997}, \cite{WardKing-1999}) the nutrient is consumed at a rate proportional to the rate of cell mitosis, namely  the  second term 
  on the right-hand side of  the first equation in \eqref{dC}.  We also refer the reader to Friedman \cite{Friedman-2004}  where a class of relevant tumor growth models  are presented and the evolution of the nutrient is given by a related equation.
  
  \subsubsection{A linear diffusion equation for the evolution of drug}
 The evolution of the drug concentration in the tumor is given by a diffusion equation of the form
  \begin{equation}
\frac{\partial W}{\partial t} = \Grad \cdot (\nu_2 \Grad W) -  \left(\mu_1 G_1(W)P + \mu_2  G_2(W) Q\right) W,\nonumber
\end{equation}
with $G_1(\cdot), G_2(\cdot)$ smooth functions.

Assuming that $\nu_2$ is constant this equation (cf. Zhao \cite{Zhao-2010}) becomes 
\begin{equation}\label{dW}
\frac{\partial W}{\partial t} = \nu_2 \Delta W -  \left(\mu_1 G_1(W) P + \mu_2 G_2(W) Q \right)W.
\end{equation}
This equation describes the diffusion of the drug within the tumor region. The second term of the right-hand side of \eqref{dW} represents the drug consumption, the constants $\mu_1, \mu_2$ are two positive constants which can be viewed as a measure of the drug effectiveness. We refer the reader to Ward and King \cite{WardKing-1997, WardKing-2003}
and Zhao \cite{Zhao-2010} for further comments.

\smallskip

The total density of the mixture is denoted by $\vr_f$ and is given by\\
\begin{equation}
 \vr_f = P+Q+D = Constant.
\label{density}
\end{equation}

Adding \eqref{dP}-\eqref{dD} and taking into consideration \eqref{density} we arrive at the following relation, which represents an additional constraint\\
 \begin{align}
\rho_{f}\Div{\vc{v}}= \vc{G_P} + \vc{G_Q}+\vc{G_D}
=
K_{B} C P-K_{R}D.
\label{divcon}
\end{align}
Our aim is to study the system \eqref{pressure2}-\eqref{divcon} in a spatial domain $\Omega_t$, with a boundary $\Gamma =\partial\Omega_t$ varying in time. 

\subsection{Boundary behavior}
The boundary of the domain $\Omega_t$ occupied by the tumor is described by means of a given velocity $\vc{V}(t, \vc{x}),$ where $t \ge 0$ and $\vc{x} \in \R^3.$ More precisely, assuming $\vc{V}$ is regular, we solve the associated system of differential equations
\begin{equation}
\frac{d}{dt} \vc{X}(t, \vc{x}) = \vc{V}(t, \vc{X}), \,\, t > 0, \,\, \vc{X}(0,\vc{x}) = \vc{x}, \nonumber
\end{equation}
and set
\begin{equation}
\begin{cases}
\!\!\! \! & \Omega_{\tau} = \vc{X}(\tau, \Omega_0), \,\, \mbox{where}\,\, \Omega_0 \subset \R^3\,\, \mbox{is a given domain,}\\
\!\! \!\! & \Gamma_{\tau} = \partial \Omega_{\tau}, \,\, \mbox{and} \,\, Q_{\tau} = \left\{(t,x) | t \in (0,\tau), x\in  \Omega_{\tau}\right\}.
\end{cases}\nonumber
\end{equation}
\smallskip

 The model is closed by giving boundary conditions on the (moving) tumor boundary $\Gamma_{\tau}.$ 
More precisely, we assume that the boundary $\Gamma_{\tau}$ is impermeable,  meaning
\begin{equation}
 (\vc{v} - \vc{V}) \cdot \vc{n}|_{\Gamma_{\tau}} = 0, \,\, \mbox{for any}\,\, \tau \ge 0. \label{BC1}
 \end{equation}
In addition, for {\em viscous} fluids, Navier proposed the boundary condition of the form
\begin{equation}
[\mathbb{S} \vc{n}]_{\mbox{tan}}|_{\Gamma_{\tau}} = 0, \label{BC2}
\end{equation}
with $\mathbb{S}$ denoting the viscous stress tensor which in this context is assumed to be determined through Newton's rheological law
$$
\mathbb{S} = \mu \Big( \Grad \vc{v} + \Grad^{\perp} \vc{v} - {2 \over 3} \Div
\vc{v} \mathbb{I} \Big) + \xi \Div \vc{v} \mathbb{I},
$$
where $\mu> 0$, $\xi \geq 0$ are respectively the shear and bulk
viscosity coefficients. Condition \eqref{BC2} namely says that the tangential component of the normal 
viscous stress vanishes on $\Gamma_{\tau}.$
\smallskip

  
The concentrations of the nutrient and the drug on the boundary satisfy the conditions:
\begin{equation}
C(x,t)|_{\Gamma_t} = 0,\,\,\, W(x,t)|_{\Gamma_t} = 0  \label{BC3}.
\end{equation} 

Finally, the problem \eqref{dP}-\eqref{BC3} is supplemented by the initial conditions
\begin{equation}
\begin{cases}
& \!\!\!\!  P(0, \cdot) = P_0, \,\, Q(0, \cdot) = Q_0, \,\, D(0, \cdot) = D_0, \\
& \!\!\!\!  C(0, \cdot) = C_{0}\le {\bar C}, \,\, W(0, \cdot) = W_{0}\,\,\, \text{in} \,\,\, \Omega_0.
\end{cases}
 \label{IC}
\end{equation}
Our main goal is to show the existence of global in time weak solutions to  \eqref{pressure2}-\eqref{IC} for any finite energy  initial data. Related works on the mathematical analysis of cancer  models have been presented by Zhao \cite{Zhao-2010} based on the farmework introduced  by Friedman {\em et al.}  \cite{Friedman-2004}, \cite{ChenFriedman-2013}. The analysis in \cite{Friedman-2004}, \cite{ChenFriedman-2013} yields existence and uniqueness of solution  to a related model in the radial symmetric case for a small time interval $[0,T].$ The analysis in \cite{Zhao-2010} treats a parabolic-hyperbolic free boundary problem and provides a unique global solution in the radially symmetric case. In the forth mentioned articles  the tumor tissue is assumed to be a porous medium and the velocity field is  determined by  Darcy's Law
$$\vc{v} = - \Grad_x \sigma \,\, \mbox{in} \,\, \Omega(t).$$

In  \cite{DT-MixedModel-2013}, Donatelli and Trivisa establish the global existence of weak solutions to a nonlinear system modeling tumor growth in a general moving domain $\Omega_t \subset \R^3$ without any symmetry assumption and for finite large initial data. 
The article \cite{DT-MixedModel-2013} is according to our knowledge the first article treating the problem in a general setting. In \cite{DT-VariableDensity-2014} the same authors establish the global existence of weak solutions to a nonlinear system for tumor growth in the case of variable total density of cells within a cellular medium.

The present article extends earlier results in a variety of ways. First the effect of drug application is being considered within a moving domain in $\R^3$ without any symmetry considerations. Second, the transport equations  are rather general capturing more effectively the biological  setting.  Our framework relies on biologically grounded principles ${\bf [P1]-[P7]},$ which are motivated by experiments performed by Roda {\em et al.} \cite{Roda-etal-2011} \cite{Roda-etal-2012A}, \cite{Roda-etal-2012}  and provide  a  description of the dynamics of the population of cells within the tumor.

\par\smallskip

We establish the global existence of  weak solutions to \eqref{pressure2}-\eqref{IC} on time dependent domains, supplemented with slip boundary conditions. In the center of our approach lie the so-called 
{\em generalized penalty methods}  typically suitable for treating partial slip, free surface, contact and related boundary conditions in viscous flow analysis and simulations.
As has been seen in earlier works, (cf. Carey and Krishnan \cite{CareyKrishnan-1982}, \cite{CareyKrishnan-1984},\cite{CareyKrishnan-1985}, Donatelli and Trivisa \cite{DT-MixedModel-2013}) penalty methods  provide an additional weakly enforce constraint in the problem.
This form of boundary penalty approximation appeared  by Courant in  \cite{Courant-1956}, in the context of  slip conditions for stationary incompressible fluids by Stokes and Carrey in \cite{StokesCarey-2011}, and more recently in a series of articles (cf. \cite{DT-MixedModel-2013}, \cite{DT-VariableDensity-2014}, \cite{FeireislNS-2011}, \cite{FeireislKNNS-2013}).
The existence theory for the barotropic Navier-Stokes
system on {\em fixed} spatial domains in the framework of weak solutions was developed in the seminal work of Lions \cite{Lions-1998}. 

\subsection{Outline}
The paper is organized as follows: Section \ref{S1} presents the motivation, modeling  and   introduces the necessary preliminary material. Section \ref{S2}  provides a weak formulation of the problem and states the main result. Section \ref{S3}  is devoted to the penalization problem and to the construction  of a suitable approximate scheme. The central component of the approximating procedure is the addition of  a singular forcing term
\begin{equation*}
\frac{1}{\varepsilon}  \int_{\Gamma_t} (\vc{v}-\vc{V}) \cdot {\bf n} \vc{\varphi} \cdot \vc{n} dS_x ,\,\,\, \varepsilon>0\,\, \mbox{small}, 
\end{equation*}
penalizing the normal component of the velocity on the boundary of the tumor domain
in the variational formulation of Brinkman's equation. We remark that  applying a penalization method to the slip boundary conditions is extremely delicate. Unlike for no-slip boundary condition, where the fluid velocity coincides with the field $V$ outside $\Omega_{\tau}$, it is only its normal component  $\vc{v} \cdot {\bf n}$ that can be controlled in the case of slip.  In order to treat the moving boundary, additional penalizations on the viscosity and diffusion parameters are required.
In Section \ref{S4}  we give a sketch on the existence of solutions of the penalization scheme in the healthy tissue.
In Section \ref{S5} we collect all the uniform bounds satisfied by the solution of the penalization scheme. In Section \ref{S6}, the singular limits for $\varepsilon \to 0, \omega \to 0$ are performed successively.
A key part  in the penalization limit is to get rid of the terms supported in the healthy tissue part $((0,T)\times B)\backslash Q_{T}$.  The main issue is to describe the evolution of the interface $\Gamma_{\tau}.$ To that effect we employ elements from the so-called {\em level set method}  (cf.\ Osher and Fedwik \cite{OshFed03}).

\section{Weak formulation and main results}\label{S2}
\subsection{Weak solutions}
\begin{definition}\label{D2.1}
 We say that $(P, Q, D, \vc{v}, C, W)$ is a weak solution of problem
(\ref{pressure2})- \eqref{Gd}, \eqref{dC}, \eqref{dW}, \eqref{density} supplemented with boundary data satisfying
(\ref{BC1})-(\ref{BC3})  and initial data $(P_0,  Q_0, D_0, C_0, W_0)$
satisfying (\ref{IC}) provided that the following hold:
\vspace{0.1in}

$\bullet$ $(P,Q, D) \ge 0$ represents a weak solution of \eqref{dP}-\eqref{dQ}-\eqref{dD} on $(0,\infty)\times\Omega_{\tau}$, i.e., for any test function $\varphi \in C^{\infty}_c (([0,T)\times \mathbb{R}^3), T>0$ 
the  following integral relations hold
\smallskip


\begin{equation} 
\left.
\begin{array}{l}
\displaystyle{\int_{\Omega_{\tau}}  P \varphi(\tau,\cdot) \, dx  - \int_{\Omega_0}  P_0 \varphi(0,\cdot)dx = }\\
\hspace{1.5cm}\displaystyle{\int_0^{\tau} \!\!\int_{\Omega_t} \left( P \partial_t \varphi + P \vc{v} \cdot \Grad_x \varphi + \vc{G_P}  \varphi(t, \cdot) \right) dx dt}, \\ \\
\displaystyle{\int_{\Omega_{\tau}}  Q \varphi(\tau,\cdot) \, dx - \int_{\Omega_0}  Q_0 \varphi(0,\cdot)dx}  = 
 \\
\hspace{1.5cm}\displaystyle{\int_0^{\tau} \!\!\int_{\Omega_t} \left( Q \partial_t \varphi + P \vc{v} \cdot \Grad_x \varphi + \vc{G_Q}  \varphi(t, \cdot) \right) dx dt,}\\ \\
 \displaystyle{\int_{\Omega_{\tau}} D \varphi(\tau,\cdot) \, dx  - \int_{\Omega_0}  D_0 \varphi(0,\cdot)dx  =} \\
 \hspace{1.5cm}\displaystyle{\int_0^{\tau} \!\!\int_{\Omega_t} \left( D \partial_t \varphi + D \vc{v} \cdot \Grad_x \varphi + \vc{G_D}  \varphi(t, \cdot) \right) dx dt}.
\end{array}
\right\}
\tag {\bf{I}}
 \label{w-Da}
\end{equation}
In particular, 
$$P \in L^p([0,T]; \Omega_{\tau}), \,\, Q \in L^p([0,T]; \Omega_{\tau}), \,\,D \in L^p([0,T]; \Omega_{\tau}) \,\, \mbox{for all}\,\, p \ge 1. $$ 

We remark that in  the weak formulation, it is convenient that the equations \eqref{dP}-\eqref{dD} hold in the whole space $\mathbb{R}^3$ provided that the densities $(P,Q,D)$ are extended to be zero outside the tumor domain.
\smallskip

$\bullet$ Brinkman's equation \eqref{pressure2} holds in the sense of distributions, i.e., for any test function 
$\vc{\varphi} \in C^{\infty}_c(\mathbb{R}^3; \mathbb{R}^3)$ satisfying 
$$ \vc{\varphi} \cdot  \vc{n}|_{\Gamma_{\tau}} = 0\,\, \mbox{for any}\,\, \tau \in [0,T],$$ 
the following integral relation holds
\begin{equation}
\int_{\Omega_\tau} \sigma \Div \vc{\varphi} \, dx-   
 \int_{\Omega_\tau} \left( \mu\Grad_x \vc{v}  : \Grad_x \vc{\varphi} + \frac{\mu}{K} \vc{v} \vc{\varphi}   \right) dx=0.
\label{w-pressure2}
\end{equation}

All quantities in \eqref{w-pressure2} are required to be integrable, so in particular, 
$$\vc{v} \in W^{1,2}(\mathbb{R}^3;\mathbb{R}^3),$$
and
$$ (\vc{v - V}) \cdot \vc{n}(\tau, \cdot)|_{\Gamma_{\tau}}=0\,\, \mbox{for a.a.}\,\, \tau \in [0,T].$$

\smallskip

$\bullet$ $C \geq 0$ is a weak solution of \eqref{dC}, i.e.,  for any test function $\varphi \in C^{\infty}_c ([0,T)\times \mathbb{R}^3), T>0$ 
the  following integral relations hold

\[
\int_{\Omega_{\tau}}  C \varphi(\tau,\cdot) \, dx - \int_{\Omega_0}  C_0 \varphi(0,\cdot)dx  = \int_0^{\tau} \!\!\int_{\Omega_{t}}  C \partial_t \varphi dx dt   -
\]
\[
\int_{0}^{\tau} \!\!\int_{\Omega_t} \nu_1  \Grad_x C\cdot \Grad_x \varphi dx dt 
- \int_0^{\tau} \!\!\int_{\Omega_t}  \left(K_1 K_P CP + K_2 K_Q(\bar{C}-C)Q\right)  C  \varphi dx dt. 
\]
\smallskip

$\bullet$ $W \geq 0$ is a weak solution of \eqref{dW}, i.e.,  for any test function $\varphi \in C^{\infty}_c ([0,T)\times \mathbb{R}^3), T>0$ 
the  following integral relations hold

\[
\int_{\Omega_{\tau}}  W \varphi(\tau,\cdot) \, dx - \int_{\Omega_0}  W_0 \varphi(0,\cdot)dx  = \int_0^{\tau} \!\!\int_{\Omega_{t}}  W \partial_t \varphi dx dt   -
\]
\[
\int_{0}^{\tau} \!\!\int_{\Omega_t} \nu_2  \Grad_x W \cdot \Grad_x \varphi dx dt 
- \int_0^{\tau} \!\!\int_{\Omega_t}   \left(\mu_1 G_1(W) P + \mu_2 G_2(W) Q \right) Wdx dt. 
\]
\end{definition}

The main result of the article now follows. 

\begin{theorem}\label{T2.2}
Let $\Omega_0 \subset \mathbb{R}^3$ be a bounded domain of class $C^{2+\nu}$  and let 
$$\vc{V} \in C^1([0,T]; C^3_c(\mathbb{R}^3; \mathbb{R}^3))$$
 be given. Let the initial data satisfy
 $$P_0 \in L^{p}(\mathbb{R}^3), \,\, Q_0 \in L^{p}(\mathbb{R}^3),\,\, D_0 \in L^{p}(\mathbb{R}^3),\,\, \mbox{for all}\,\, p\ge 1$$
 and 
$$  C_0 \in L^{2}(\mathbb{R}^3)\cap L^{\infty}(\mathbb{R}^3),\,\, W_0 \in L^{2}(\mathbb{R}^3)\cap L^{\infty}(\mathbb{R}^3), $$ 
$$with \,\, (P_0, Q_0, D_0, C_{0}, W_0) \ge 0, \,\,\,  (P_0, Q_0, D_0, C_0, W_0) \not\equiv 0,$$
$$ P_0+Q_0+D_0=\varrho_{f},\quad  (P_0, Q_0, D_0, C_0, W_0)|_{\mathbb{R}^3 \setminus \Omega_0} =0. $$

Then the problem \eqref{pressure2}-\eqref{Gd}, \eqref{dC}, \eqref{dW}-\eqref{divcon} with initial data \eqref{IC}  and boundary data \eqref{BC1}-\eqref{BC3} admits a weak solution in the sense specified in Definition \ref{D2.1}.
\end{theorem}

\section{Penalization}\label{S3}
\subsection{General strategy}
The main ingredients of our strategy can be formulated as follows:
\begin{itemize}
\item Our approach relies on {\em penalization} of the boundary behavior, diffusion and viscosity in the weak formulation. A penalty approach to slip conditions for {\em stationary incompressible flow} was proposed by Stokes and Carey \cite{StokesCarey-2011} In the present setting, the variational (weak) formulation of the Brinkman equation is supplemented by a singular forcing term
\begin{equation} 
\frac{1}{\varepsilon} \int_{\Gamma_t} (\vc{v}-\vc{V}) \cdot {\bf n} \vc{\varphi} \cdot \vc{n} dS_x,\,\,\, \varepsilon>0\,\, \mbox{small}, \label{penalty}
\end{equation}
penalizing the normal component of the velocity on the boundary of the tumor domain.

\item In addition to  \eqref{penalty}, we introduce a {\em variable} shear viscosity coefficient $\mu = \mu_{\omega},$ as well as a {\em variable} diffusions $\nu_i={\nu_i}_{\omega}, i=1,2 $ with $\mu_{\omega}, {\nu_i}_{\omega}$  vanishing outside the tumor domain and remaining positive within the tumor domain.

\item In constructing the approximating problem we employ the variables $\varepsilon$ and $\omega.$ 
Keeping $\varepsilon$ and $\omega$ fixed, we solve the modified problem in a (bounded) reference domain $B \subset \mathbb{R}^3$ chosen in such way that 
$$\bar{\Omega}_{\tau} \subset B \,\, \mbox{for any}\,\, \tau \ge 0.$$  

\item We take the initial densities $(P_0, Q_0, D_0)$ vanishing outside $\Omega_0,$  and letting the penalization  $\varepsilon \to 0$ for fixed $\omega > 0 $ we obtain a ``two-phase" model consisting of the {\em tumor region}  and the {\em healthy tissue} separated by impermeable boundary.  We show that the densities vanish  in part of the reference domain, specifically on $((0,T) \times B) \setminus Q_T.$ 

\item We let first the penalization $\varepsilon$ vanish and next we perform the limit $\omega \to 0.$

\end{itemize}
\par\smallskip

\subsection{Penalization scheme}
As typical in time dependent  regimes the penalization can be applied to the interior of a fixed reference domains. In that way we obtain at the limit a two-phase model consisting of the tumor region  $\Omega_\tau$ and a healthy tissue $B\setminus \Omega_{\tau}$ separated by an impermeable interface
$\Gamma_{\tau}.$  As a result an extra stress is produced acting on the fluid by its complementary part outside $\Omega_{\tau}.$

We  choose  $R>0$ such that 
\begin{equation}
\vc{V}|_{[0,T] \times \{|\vc{x}| > R\} }=0, \,\,\, \bar{\Omega}_0 \subset \{|\vc{x}| <R\}\nonumber
\end{equation}
and we take as the reference fixed domain 
$$B = \{|\vc{x}| < 2R\}.$$

In order to eliminate this extra stresses we introduce  a variable shear viscosity coefficient $\mu = \mu_{\omega}(t, \vc{x})$ where, $\mu=\mu_{\omega}$ remains strictly positive in $Q_{T}$ but vanishes in $Q_{T}^{c}$ as $\omega \to 0$, namely $ \mu_{\omega}$
is taken such that 
$$ \mu_{\omega} \in C^{\infty}_c \left([0,T] \times \mathbb{R}^3\right), \,\,\, 0<\underline{{\mu}}_{\omega}\le \mu_{\omega}(t,\vc{x}) \le \mu\,\, \mbox{in}\,\, [0,T]\times B,$$
\begin{equation}
\mu_{\omega}=
\begin{cases}
\mu={\mathrm const}>0 & \text{in $Q_{T}$}\\
\mu_{\omega}\to 0 & \text{a.e. in $((0,T)\times B)\backslash Q_{T}$}
\end{cases}\nonumber
\end{equation}
and a variable diffusion coefficients of the nutrient and the drug $\nu_{i} = {\nu_i}_{\omega}(t, \vc{x}),$ where $\nu_{i}={\nu_i}_{\omega}$, $i=1,2$ remain strictly positive in $Q_{T}$ but vanishes in $Q_{T}^{c}$ as $\omega \to 0$, namely $ {\nu_{i}}_{\omega}$
are taken such that 
$$ {\nu_i}_{\omega} \in C^{\infty}_c \left([0,T] \times \mathbb{R}^3\right), \,\,\, 0<\underline{{\nu}_i}_{\omega}\le {\nu_i}_{\omega}(t,\vc{x}) \le \nu_i\,\, \mbox{in}\,\, [0,T]\times B,$$
\begin{equation}
{\nu_i}_{\omega}=
\begin{cases}
{\nu}_i={\mathrm const}>0 & \text{in $Q_{T}$}\\
{\nu_i}_{\omega}\to 0 & \text{a.e. in $((0,T)\times B)\backslash Q_{T}.$}
\end{cases}\nonumber
\end{equation}
Finally we modify the initial data for $P$, $Q$, $D$, $C$ and $W$ so that the following set of relations denoted by {\text {\bf (IC-p)}} read

\begin{equation}
\begin{cases}
& \!\!\! P_{0}=P_{0,\omega,\e}=P_{0,\omega},\ P_{0,\omega}\geq 0,\ P_{0,\omega}\not \equiv 0,\ P_{0,\omega}|_{\R^{3}\backslash\Omega_{0}}=0,\  {\displaystyle \int_{B} P_{0,\omega}^{p}dx\leq c},\\ \\
& \!\!\! Q_{0}=Q_{0,\omega,\e}=Q_{0,\omega},\ Q_{0,\omega}\geq 0,\ Q_{0,\omega}\not \equiv  0,\ Q_{0,\omega}|_{\R^{3}\backslash\Omega_{0}}=0,\  {\displaystyle\int_{B} Q_{0,\omega}^{p}dx \leq c}, \\ \\
& \!\!\! D_{0}=P_{0,\omega,\e}=D_{0,\omega},\  D_{0,\omega}\geq 0,\ D_{0,\omega}\not \equiv  0,\ D_{0,\omega}|_{\R^{3}\backslash\Omega_{0}}=0,\  {\displaystyle\int_{B} D_{0,\omega}^{p}dx\leq c},\\ \\
& \!\!\! C_{0}=C_{0,\omega,\e}=C_{0,\omega},\ C_{0,\omega}\geq 0,\ C_{0,\omega}\not\equiv 0,\ C_{0,\omega}|_{\R^{3}\backslash\Omega_{0}}=0, \ {\displaystyle\int_{B} C_{0,\omega}^{p}dx\leq c},\\ \\
& \!\!\! W_{0}=W_{0,\omega,\e}=W_{0,\omega},\ W_{0,\omega}\geq 0,\ W_{0,\omega}\not\equiv 0,\ W_{0,\omega}|_{\R^{3}\backslash\Omega_{0}}=0, \ {\displaystyle\int_{B} W_{0,\omega}^{p}dx\leq c}, \nonumber
\end{cases}
\end{equation}
for all  $p\geq 1$.

The weak formulation of the penalized problem reads:

\begin{itemize}
\item The integral relations \eqref{w-Da} in Definition \eqref{D2.1} hold true
for any $\tau \in [0,T]$ and $x \in B$ and any test function $\varphi \in C_c^{\infty}([0,T] \times \mathbb{R}^3),$ and for 
$\vc{G_{P_{\omega, \e}}}, \vc{G_{Q_{\omega, \e}}}, \vc{G_{D_{\omega, \e}}}$ given in \eqref{Gp}-\eqref{Gd}, namely
\begin{equation} 
\left.
\begin{array}{l}
 \hspace{0.5cm}\displaystyle{\int_{B}  P_{\omega,\e} \varphi(\tau,\cdot) \, dx - \int_{\Omega_0}  P_0 \varphi(0,\cdot)dx=}\\
 \hspace{1.3cm}\displaystyle{
 \int_0^{\tau} \!\!\int_{B} \left( P_{\omega,\e} \partial_t \varphi + P_{\omega, \e} \vc{v}_{\omega, \e} \cdot \Grad_x \varphi + \vc{G_{P_{\omega, \e}}}  \varphi(t, \cdot) \right) dx dt}, \\ \\
 \hspace{0.5cm}\displaystyle{ \int_{B}  Q_{\omega, \e} \varphi(\tau,\cdot) \, dx - \int_{\Omega_0}  Q_0 \varphi(0,\cdot)dx} =\\
\hspace{1.3cm}\displaystyle{\int_0^{\tau} \!\!\int_{B} \left( Q_{\omega, \e} \partial_t \varphi + P_{\omega, \e} \vc{v}_{\omega, \e} \cdot \Grad_x \varphi + \vc{G_{Q_{\omega, \e}}} \varphi(t, \cdot) \right) dx dt},\\ \\
 \hspace{0.5cm} \displaystyle{ \int_{B} D_{\omega, \e} \varphi(\tau,\cdot) \, dx   - \int_{\Omega_0}  D_0 \varphi(0,\cdot)dx}= \\
\hspace{1.3cm}\displaystyle{\int_0^{\tau} \!\!\int_{B} \left( D_{\omega, \e} \partial_t \varphi + D_{\omega, \e} \vc{v}_{\omega, \e} \cdot \Grad_x \varphi + \vc{G_{D_{\omega, \e}}}  \varphi(t, \cdot) \right) dx dt }.
\end{array}
\right\}
\tag {\bf{Ip}}
 \label{w-D}
\end{equation}

\item The weak formulation for the penalized Brinkman's equation reads
\begin{equation}
\begin{split}
\int_{B}\sigma_{\omega, \e} \Div \vc{\varphi}dx
-  \int_B \big(\mu_{\omega} \Grad_x \vc{v}_{\omega, \e}  &: \Grad_x \vc{\varphi} + \frac{\mu_{\omega}}{K} \vc{v}_{\omega, \e} \vc{\varphi}   \big) dx  \\
+ \frac{1}{\varepsilon}  \int_{\Gamma_t} ((\vc{V} - \vc{v}_{\omega, \e}) &\cdot \vc{n} \vc{\varphi} \cdot \vc{n} ) dS_x =0
\end{split}
\label{w-p-mom}
\end{equation}
for any test function $\vc{\varphi} \in C_c^{\infty}(B; \mathbb{R}^3),$
where $\vc{v}_{\omega, \e} \in W_0^{1,2}(B; \mathbb{R}^3),$
and  $\vc{v}_{\omega, \e}$ satisfies the no-slip boundary condition 
\begin{equation}
\vc{v}_{\omega, \e}|_{\partial B} = 0\,\, \mbox{in the sense of traces}. \label{no-slip}
\end{equation} 

\item The weak formulation for $C_{\omega,\e}$ is as follows,
\[
\int_B C_{\omega, \e} \varphi(\tau,\cdot) \, dx - \int_{\Omega_0}  C_0 \varphi(0,\cdot)dx  = 
\]
\begin{equation}
\int_0^{\tau} \!\!\int_B  C_{\omega, \e} \partial_t \varphi dx dt   -
\int_{0}^{\tau} \!\!\int_B {\nu_1}_{\omega}  \Grad_x C_{\omega, \e} \cdot \Grad_x \varphi dx dt
\label{weakc}
\end{equation}
\[ 
- \int_0^{\tau} \!\!\int_B  \left(K_1 K_P C_{\omega,\e}P_{\omega,\e} + K_2 K_Q(\bar{C}-C_{\omega,\e})Q\right)  C_{\omega, \e}\varphi dx dt,
\]
for any test function $\varphi \in C_c^{\infty}([0,T] \times \mathbb{R}^3)$ and $C_{\omega,\e}$ satisfies the boundary conditions 
\begin{equation}
 C_{\omega, \e}|_{\partial B} = 0\,\, \mbox{in the sense of traces}. \label{neumann}
\end{equation} 
\item The weak formulation for $W_{\omega,\e}$ is as follows,
\[
\int_B W_ {\omega, \e} \varphi(\tau,\cdot) \, dx - \int_{\Omega_0}  W_0 \varphi(0,\cdot)dx  = 
\]
\begin{equation}
\int_0^{\tau} \!\!\int_B  W_{\omega, \e} \partial_t \varphi dx dt   -
\int_{0}^{\tau} \!\!\int_B {\nu_2}_{\omega}  \Grad_x W_{\omega, \e} \cdot \Grad_x \varphi dx dt
\label{weakw}
\end{equation}
\[ 
- \int_0^{\tau} \!\!\int_B  \left(\mu_1 G_1(W_{\omega,\e}) P_{\omega,\e} + \mu_2 G_2(W_{\omega,\e}) Q_{\omega,\e} \right) 
W_{\omega, \e}  \varphi dx dt, 
\]
for any test function $\varphi \in C_c^{\infty}([0,T] \times \mathbb{R}^3)$ and $W_{\omega,\e}$ satisfies the boundary conditions 
\begin{equation}
W_{\omega, \e}|_{\partial B} = 0\,\, \mbox{in the sense of traces}. \label{neumann}
\end{equation}
\end{itemize}

Here, $\varepsilon$ and $\omega$ are positive parameters. 

\section{Existence of Approximate Solutions within B}
\label{S4}
The construction of the approximate solutions $$(P_{\omega, \e}, Q_{\omega, \e}, D_{\omega, \e}, v_{\omega, \e}, C_{\omega, \e}, W_{\omega, \e})$$ within  the fixed reference domain $B$  relies
\begin{enumerate}
\item[--]  on the regularization of the three transport equations \eqref{dP}-\eqref{dD} with the aid of an artificial viscosity parameter $\eta$ transforming the three transport (hyperbolic) equations into parabolic partial differential equations, and 
\item[--]  on  the use of the so-called Faedo Garlerkin approximations on Brink-man's equation which involves replacing  \eqref{pressure2} by an integral relation. The approximation at this level involves a  parameter $n,$ denoting the dimension of the basis used in this process.
\end{enumerate}
Given the approximate velocity, and the nutrient and drug  concentrations  one solves  the three parabolic equations corresponding to \eqref{dP}-\eqref{dD} via a fixed point argument. Next, one solves the diffusion equations obtaining the nutrient and the drug concentrations.

The loop closes by performing a fixed point argument on the integral form of Brinkman's equation yielding the approximate velocity. The existence of the approximate solutions $\{P_{\omega, \e}, Q_{\omega, \e}, D_{\omega, \e}, v_{\omega, \e}, C_{\omega, \e}, D_{\omega, \e} \}$ within $B$ is established by letting $n \to \infty$ and $\eta \to 0$ in the spirit of the analysis in \cite{Donatelli-Trivisa-2006}.

We emphasize at this point that by adding the three parabolic equations of the approximate  cell densities corresponding to \eqref{dP}-\eqref{dD} one obtains a parabolic equation for the sum of cell densities $[P+Q+D]_{\{\omega, \e, \eta, n\}}.$ At this point we omit the indices  for simplicity in the presentation.

We recall that  in   $\Omega_{\tau}$, \eqref{density}  holds, namely  $P+Q+D= \varrho_{f}$.  A simple argument shows that this sum is constant within the fixed reference domain $B$  as well. At this level one can argue by contradiction, namely assume that
$$P+Q+D= R(t) \not = \varrho_{f}$$ and write the equation verified by $R(t)$ which is the following  linear parabolic equation
\begin{equation}
\partial_t R(t) + \vc{v} \nabla R(t) = \eta \Delta R(t) +  \frac{1}{\varrho_{f}} [K_B CP - K_RD] [\varrho_{f}-R(t)].
\label{R1}
\end{equation}
supplemented with the initial data 
\begin{equation}
R(0)=\varrho_{f}.
\label{R2}
\end{equation}
Applying Gronwall's inequality now yields uniqueness of solutions for \eqref{R1}-\eqref{R2}. Observing now, that $R(t)=\varrho_{f}$ is a solution of \eqref{R1}-\eqref{R2} leads to contradiction.

\section{Uniform bounds}\label{S5}
In this section we collect all the uniform bounds satisfied by the solutions of the penalization schemes defined in the Section \ref{S3}. Let us mention that we will denote by $c$ a constant that depends on the initial data \eqref{IC}, the boundary conditions \eqref{BC2}-\eqref{BC3}, $\vr_{f}$, $\|C_{\omega,\e}\|_{L^{\infty}_{t,x}}$, $\|W_{\omega,\e}\|_{L^{\infty}_{t,x}}$.
From the previous section we get that 
\begin{equation}
0\leq P_{\omega,\e},  Q_{\omega,\e},  D_{\omega,\e}\leq \vr_{f}\quad \text{in $[0,T]\times B$},
\label{u1}
\end{equation}
this entails that for any $p\geq 1$
\begin{equation}
P_{\omega,\e},\  Q_{\omega,\e},\  D_{\omega,\e} \quad \text {are uniformly bounded in $L^{p}([0,T]\times B)$}.
\label{u2}
\end{equation}
Since the nutrient $C_{\omega,\e}$  and the drug concentration satisfy a parabolic equation, by a standard application of  the maximum principle \cite{AS67}  we have that almost everywhere in $B\times (0,T)$
\begin{equation}
C_{\omega, \e}(x,t)\in L^{\infty}([0,T]\times B).
\label{mpC}
\end{equation}
\begin{equation}
W_{\omega, \e}(x,t)\in L^{\infty}([0,T]\times B).
\label{mpW}
\end{equation}
Now,  by multiplying \eqref{dC} by  $C_{\omega,\e}$,  by integrating by parts and by taking into account \eqref{u1}, \eqref{u2}, \eqref{mpC} we get that $C_{\omega,\e}$ satisfies the following energy estimate,
\begin{equation}
\frac{\partial}{\partial t}\int_{B}\frac{1}{2}C^{2}_{\omega, \e}dx + \int_{B} {\nu_{1}}_{\omega} |\Grad_{x}C_{\omega, \e}|^{2}dx\leq c\int_{B}C^{2}_{\omega, \e}dx, 
\label{u3}
\end{equation}
similarly, taking into account that $G_{1}$ and $G_{2}$ are smooth functions we have also
\begin{equation}
\frac{\partial}{\partial t}\int_{B}\frac{1}{2}W^{2}_{\omega, \e}dx +\int_{B} {\nu_{2}}_{\omega}|\Grad_{x}W_{\omega, \e}|^{2}dx\leq c\int_{B}W^{2}_{\omega, \e}dx.
\label{u3w}
\end{equation}
As a consequence of \eqref{mpC}, \eqref{mpW}, \eqref{u3}, \eqref{u3w} we get the following uniform bounds  with respect to $\varepsilon$, $\omega$.\\
\begin{equation}
\|C_{\omega,\e}\|_{L^{2}_{t}L^{2}_{x}}+ \| {\nu_{1}}_{\omega} \Grad C_{\omega,\e}\|_{L^{2}_{t}L^{2}_{x}}\leq c,
\label{bc}
\end{equation}\\
\begin{equation}
\|W_{\omega,\e}\|_{L^{2}_{t}L^{2}_{x}}+  \| {\nu_{2}}_{\omega} \Grad W_{\omega,\e}\|_{L^{2}_{t}L^{2}_{x}}\leq c,
\label{bw}
\end{equation}\\
where $L^{q}_{t}L^{p}_{x}$ stands for $L^{q}(0,T;L^{2}(B)\!)$.
By combining  \eqref{u2}, \eqref{mpC}  with \eqref{divcon} we have that
\begin{equation}
\Div\vc{v}_{\omega,\e}={\bf G}, \qquad \text{with ${\bf G}\in L^{\infty}(0,T;L^{p}(B))$,\quad $p>1$}.
\label{bdiv}
\end{equation}
Next, by applying regularity theory concerning the divergence  equation in Sobolev spaces (see Lemma 2.1.1 (a) in \cite{Sohr-2001} or Remark 3.19 in \cite{Novotny-Stras-2004}, for more details see also \cite{DT-MixedModel-2013}) we end up with 
\begin{equation}
 \|\nabla\vc{v}_{\omega, \e}\|_{L^{p}_{x}}\leq c\|{\bf G}\|_{L^{p}_{x}} \qquad p>1.
\label{bgradv}
\end{equation}
Since the vector field $\vc{V}$ vanishes on the boundary of the reference domain $B$ it may be used as a test function in the weak formulation of the Brinkman's equation  for the  penalized problem \eqref{w-p-mom}, namely

\begin{equation}
\begin{split}
\int_B \sigma_{\omega,\e}\Div\vc{V}dx&- \int_B \big(\mu_{\omega} \Grad_x \vc{v}_{\omega,\e}  : \Grad_x \vc{V} + \frac{\mu_{\omega}}{K} \vc{v}_{\omega,\e} \vc{V}  \big) dx \\
&+ \frac{1}{\varepsilon}  \int_{\Gamma_t} ((\vc{V} - \vc{v}_{\omega,\e}) \cdot \vc{n} \vc{V} \cdot \vc{n} )dS_x =0.
\end{split}
 \label{w-p-momV}
\end{equation}
By combining standard computations with \eqref{w-p-momV}, the velocity field $\vc{v}_{\omega,\e}$ satisfies the following estimate,
 \begin{equation*}
 \begin{split}
\int_B  \mu_{\omega} ( |\Grad_x \vc{v}_{\omega,\e}|^2 + \frac{1}{K}|\vc{v}_{\omega,\e}|^{2})dx +\frac{1}{\varepsilon}\int_{\Gamma_t} |(\vc{v}_{\omega,\e}-\vc{V}) \cdot \vc{n}|^{2}dS \leq\\
\int_B \left(\mu_{\omega} \Grad_x \vc{v}_{\omega,\e}  : \Grad_x \vc{V} +\mu_{\omega} \vc{v}_{\omega,\e} \vc{V}  \right)dx+\int_B\sigma_{\omega,\e}\left( \Div\vc{v}_{\omega,\e}- \Div_x \vc{V} \right)dx. 
\end{split}
\end{equation*}

Since the vector field $\vc{V}$ is smooth  by means of \eqref{bdiv},   \eqref{w-p-momV}
 and by considering the weak formulation of  Brinkmann's equation for the penalized problem  \eqref{w-p-mom} in $B$ once more with a  special test function (for example by employing the multipliers technique of Lions \cite{Lions-1998} for the pressure) 
we get  the following uniform bounds with respect to $\e$, $\omega$.

\begin{equation}
\|\sigma_{\omega,\e} \|_{L^{\beta}_{x}}\leq c, \qquad 1<\beta\leq 2
\label{bpre}
\end{equation}

 \begin{equation}
\|\mu_{\omega} \vc{v}_{\omega,\e}\|_{L^{2}_{x}}+\|\mu_{\omega} \Grad\vc{v}_{\omega,\e}\|_{L^{2}_{x}}\leq c,
\label{bv}
\end{equation}\\
\begin{equation}
\int_{\Gamma_t} |(\vc{v}_{\omega,\e}-\vc{V}) \cdot \vc{n}|^{2}dS\leq c\varepsilon.
\label{bp}
\end{equation}

\section{Singular limits}\label{S6}
In this section we perform the limits of our two level penalization approximation. The first step is to keep $\omega$ fixed and let  $\varepsilon\to 0$. The main issue of this step is to get rid of the quantities that are supported by the  healthy tissue $B\backslash \Omega_{t}$. This will be done by means of the  Lemma \ref{FL} that we will prove in the section. The second and final step  is the vanishing viscosity limit $\omega\to 0$ that we  perform in Section \ref{Sv} and this completes the proof of our main result Theorem \ref{T2.2}. 
 \subsection{Vanishing penalization $\varepsilon \to 0$}
 As a consequence of the uniform bound \eqref{u2} and the equations \eqref{dP}, \eqref{dQ}, \eqref{dD}  we get that the weak solutions of our approximation system satisfy
\begin{equation}
\left.
\begin{array}{r}
P_{\omega,\e}\rightarrow P_{\omega}\\ \\
Q_{\omega,\e}\rightarrow Q_{\omega}\\ \\
D_{\omega,\e}\rightarrow D_{\omega}
\end{array}
\right\}\quad \text{in} \quad C_{\text{weak}}(0,T;L^{p}(B)),\ p\geq 1.
\label{cPQDC}
\end{equation}
From the bound \eqref{bc}, \eqref{bw} and \eqref{bv} we get
\begin{align}
C_{\omega,\e}\longrightarrow C_{\omega}
\qquad  &\text{weakly in $L^{2}(0,T;W^{1,2}_{0}(B))$}, \label{cc}\\
W_{\omega,\e}\longrightarrow W_{\omega}
\qquad &\text{weakly in $L^{2}(0,T;W^{1,2}_{0}(B))$}, 
\label{cw}\\
\vc{v}_{\omega,\e}\longrightarrow \vc{v}_{\omega} \qquad &\text{weakly in $W^{1,2}_{0}(B)$}, 
\label{cv}
\end{align}
while from \eqref{bp} we have that
\begin{equation}
(\vc{v}_{\omega, \e}-\vc{V}) \cdot \vc{n}(\tau, \cdot)\big|_{\Gamma_{\tau}}=0\quad \text{for a.a $\tau\in [0,T]$.}\nonumber
\label{cp}
\end{equation}
By combining together \eqref{u2}, \eqref{bv} and the compact embedding of $L^{2}(B)$ in $W^{-1,2}(B)$we get 
\begin{equation}
\hspace{-0.55cm}
\left.
\begin{array}{r}
P_{\omega,\e}\vc{v}_{\omega,\e}\rightarrow P_{\omega}\vc{v}_{\omega}\\ \\
Q_{\omega,\e}\vc{v}_{\omega,\e}\rightarrow Q_{\omega}\vc{v}_{\omega}\\ \\
D_{\omega,\e}\vc{v}_{\omega,\e}\rightarrow D_{\omega}\vc{v}_{\omega}
\end{array}
\right\}\  \text{weakly-($\ast$) in} \ L^{\infty}(0,T;L^{2q/q+2}(B)),\  2\leq q< 6.
\label{cPQDCv}
\end{equation}
Finally from the equations \eqref{dP}-\eqref{dD} it follows that
\begin{equation}
\hspace{-0.55cm}
\left.
\begin{array}{r}
P_{\omega,\e}\vc{v}_{\omega,\e}\rightarrow P_{\omega}\vc{v}_{\omega}\\ \\
Q_{\omega,\e}\vc{v}_{\omega,\e}\rightarrow Q_{\omega}\vc{v}_{\omega}\\ \\
D_{\omega,\e}\vc{v}_{\omega,\e}\rightarrow D_{\omega}\vc{v}_{\omega}
\end{array}
\right\}\quad\text{in} \ C_{\text{weak}}([T_{1},T_{2}];L^{2q/q+2}(B)),\  2\leq q< 6.
\label{cPQDCv2}
\end{equation}
Since the embedding of $W^{1,2}_{0}(B)$ in $L^{6}(B)$ is compact we have that
\begin{equation}
\vc{v}_{\omega,\e}\otimes \vc{v}_{\omega,\e}\rightarrow \vc{v}_{\omega}\otimes\vc{v}_{\omega}\quad \text{weakly in $L^{6q/6+q}(B)$ for any $2\leq q<6$.}
\nonumber
\end{equation}
Taking into account \eqref{u2}, \eqref{bc}, \eqref{bw} and, as before,  the compact embedding of $L^{2}(B)$ in $W^{-1,2}(B)$ we get
\begin{equation}
\hspace{-0.55cm}
\left.
\begin{array}{r}
P_{\omega,\e}C_{\omega,\e}\rightarrow P_{\omega}C_{\omega}\\ \\
Q_{\omega,\e}C_{\omega,\e}\rightarrow Q_{\omega}C_{\omega}\\ \\
D_{\omega,\e}C_{\omega,\e}\rightarrow D_{\omega}C_{\omega}\\ \\
P_{\omega,\e}W_{\omega,\e}\rightarrow P_{\omega}W_{\omega}\\ \\
Q_{\omega,\e}W_{\omega,\e}\rightarrow Q_{\omega}W_{\omega}
\end{array}
\right\}\ \text{weakly-($\ast$) in $L^{\infty}(0,T;L^{2q/q+2}(B)),$ $2\leq q < 6$.}
\label{cC}
\end{equation}
By using \eqref{bpre} and \eqref{cv} we  into the limit in the weak formulation \eqref{w-p-mom} of the Brinkman's equation we get 
\begin{equation}
\int_{B}\sigma_{\omega} \Div \vc{\varphi}dx
-  \int_B \big(\mu_{\omega} \Grad_x \vc{v}_{\omega}  : \Grad_x \vc{\varphi} + \frac{\mu_{\omega}}{K} \vc{v}_{\omega} \vc{\varphi}   \big) dx =0,
 \label{pmom}
\end{equation}
for any test function $\vc{\varphi} \in C_c^{\infty}(B; \mathbb{R}^3), \ \vc{\varphi}\cdot\vc{n}|_{B}=0.$

By using \eqref{bv}, \eqref{cPQDC}-\eqref{cC} we can pass to the limit in the weak formulations \eqref{w-D}, \eqref{weakc}, \eqref{weakw}  and we obtain 
\begin{equation}
\left.
\begin{array}{l}
\displaystyle{\int_{B}   P_{\omega} \varphi(\tau,\cdot)dx  - \int_{B}  P_0 \varphi(0,\cdot)  dx =}\\
\hspace{1.5cm}\displaystyle{ \int_0^{\tau}\! \!\int_{B} \left( P_{\omega} \partial_t \varphi + P_{\omega} \vc{v} \cdot \Grad_x \varphi + \vc{G_{P_{\omega}}}  \varphi \right) dx dt,} \\ \\
\displaystyle{\int_{B}  Q_{\omega} \varphi(\tau,\cdot)dx -  \int_B Q_0 \varphi(0,\cdot) dx =}\\ 
\hspace{1.5cm}\displaystyle{  \int_0^{\tau}\! \!\int_{B} \left( Q_{\omega} \partial_t \varphi +Q_{\omega}\vc{v} \cdot \Grad_x \varphi + \vc{G_{Q_{\omega}}}  \varphi\right) dx dt,}\\ \\
\displaystyle{\int_{B} D_{\omega} \varphi(\tau,\cdot) dx -  \int_B D_0 \varphi(0,\cdot) dx=} \\
\hspace{1.5cm}\displaystyle{\int_0^{\tau}\!\! \int_{B} \left( D_{\omega} \partial_t \varphi + D_{\omega} \vc{v} \cdot \Grad_x \varphi + \vc{G_{D_{\omega}}}  \varphi\right) dx dt.}
\end{array}
\right\}
\tag {\bf{IIp}}
 \label{p-D}
\end{equation}\\
\begin{equation}
\begin{split}
\int_B C_{\omega} \varphi(\tau,\cdot) \, dx &- \int_{\Omega_0}  C_0 \varphi(0,\cdot)dx  = \\
\int_0^{\tau} \!\!\int_B  C_{\omega} \partial_t \varphi dx dt   &-
\int_{0}^{\tau} \!\!\int_B {\nu_1}_{\omega}  \Grad_x C_{\omega} \cdot \Grad_x \varphi dx dt\\
- \int_0^{\tau} \!\!\int_B  \big(K_1 K_P C_{\omega}P_{\omega} &+ K_2 K_Q(\bar{C}-C_{\omega})Q\big)  C_{\omega}  \varphi \dx dt. 
\end{split}
\label{p-C}
\end{equation}\\

\begin{equation}
\begin{split}
\int_B W_{\omega} \varphi(\tau,\cdot) \, dx &- \int_{\Omega_0}  W_0 \varphi(0,\cdot)dx  = \\
\int_0^{\tau} \!\!\int_B  W_{\omega} \partial_t \varphi dx dt  & -
\int_{0}^{\tau} \!\!\int_B {\nu_1}_{\omega}  \Grad_x W_{\omega} \cdot \Grad_x \varphi dx dt
\\ 
- \int_0^{\tau} \!\!\int_B  \big(\mu_2 G_1(W_{\omega}) P_{\omega} &+ \mu_2 G_2(W_{\omega,\e}) Q_{\omega} \big) 
W_{\omega}  \varphi \ dx dt. 
\end{split}
\label{p-W}
\end{equation}


\subsubsection{Vanishing density terms  in the ``healthy tissue''}
The next step in the penalization limit is to get rid of the terms supported in the healthy tissue part $((0,T)\times B)\backslash Q_{T}$.  The main issue is to describe the evolution of the interface $\Gamma_{\tau}.$ To that effect we employ elements from the so-called {\em level set method}.
The {\em level set method} is a numerical method for tracking interfaces and shapes (cf.\ Osher and Fedwik \cite{OshFed03}). It turns out that the interface $\Gamma_\tau$ can be identified with a component of the  set 
$$\{ \Phi(\tau, \cdot)=0\},$$
while the set $B\setminus \Omega_{\tau}$ correspond to $\{\Phi(\tau,\cdot)>0\}$, with
 $\Phi=\Phi(t,x)$ denoting the  unique solution of the transport equation 
\begin{equation}
\partial_{t}\Phi+\Grad_{x}\Phi(t,x)\cdot\vc{V}=0,
\label{dd1}
\end{equation}
with initial data
$$\Phi_{0}(x)=
\begin{cases}
>0 & \text{for}\  x\in B\backslash\Omega_{0},\\
<0 & \text{for}\  x\in \Omega_{0}\cup(\R^{3}\backslash\overline{B}),
\end{cases}
\qquad \Grad_{x}\Phi_{0}\neq 0\  \text{on $\Gamma_{0}$}.$$
Finally,
\begin{equation}
\begin{split}
\Grad_{x}\Phi(\tau,x)&=\lambda(\tau,x)\vc{n}(x) \qquad \text{for any $x\in \Gamma_{\tau}$}\\
&\\
&\lambda(\tau,x)\geq 0 \qquad \text{for $\tau\in[0,T]$.}
\end{split}
\label{d1}
\end{equation}

First we deal with the nutrient $C_{\omega}$ and the drug concentration $W_{\omega}$.
and we prove that their are vanishing outside $\Omega_{\tau}$.

\begin{proposition}
Assume that  $C_{\omega}$, $W_{\omega}$  satisfy \eqref{p-C} and \eqref{p-W}, respectively
and that {\text {\bf (IC-p)}} holds, then
\begin{equation}
C_{\omega}(\tau, \cdot)|_{B\backslash\Omega_{\tau}}=0\qquad W_{\omega}(\tau, \cdot)|_{B\backslash\Omega_{\tau}}=0.
 \label{sdCW}
\end{equation}
\label{fondprop1}
\end{proposition}
\begin{proof}
 In order to prove \eqref{sdCW} it is enough to observe that $C_{\delta,\omega}$ and $W_{\delta,\omega}$   are solutions in $B\backslash\Omega_{\tau}$ of a parabolic equation with vanishing initial  and boundary data.
 \end{proof}
Now, thanks to the Proposition \ref{fondprop1}, the weak formulation \eqref{p-C} assumes the following form
\begin{equation}
\begin{split}
\int_{\Omega_{\tau}} C_{\omega} \varphi(\tau,\cdot) \, dx &- \int_{\Omega_0}  C_0 \varphi(0,\cdot)dx  = \\
\int_0^{\tau} \!\!\int_{\Omega_{\tau}} C_{\omega} \partial_t \varphi dx dt   &-
\int_{0}^{\tau} \!\!\int_{\Omega_{\tau}} {\nu_1}_{\omega}  \Grad_x C_{\omega} \cdot \Grad_x \varphi dx dt\\
- \int_0^{\tau} \!\!\int_{\Omega_{\tau}}  \big(K_1 K_P C_{\omega}P_{\omega} &+ K_2 K_Q(\bar{C}-C_{\omega})Q\big)  C_{\omega}  \varphi(\tau, \cdot) dx dt. 
\end{split}
\label{p-C2}
\end{equation}
while the drug concentration formulation \eqref{p-W} becomes
\begin{equation}
\begin{split}
\int_{\Omega_{\tau}} W_{\omega} \varphi(\tau,\cdot) \, dx &- \int_{\Omega_0}  W_0 \varphi(0,\cdot)dx  = \\
\int_0^{\tau} \!\!\int_{\Omega_{\tau}}  W_{\omega} \partial_t \varphi dx dt  & -
\int_{0}^{\tau} \!\!\int_{\Omega_{\tau}} {\nu_2}_{\omega}  \Grad_x W_{\omega} \cdot \Grad_x \varphi dx dt
\\ 
- \int_0^{\tau} \!\!\int_{\Omega_{\tau}}  \big(\mu_1 G_1(W_{\omega}) P_{\omega} &+ \mu_2 G_2(W_{\omega,\e}) Q_{\omega} \big) 
W_{\omega}  \varphi(\tau, \cdot) dx dt. 
\end{split}
\label{p-Ww}
\end{equation}
In order to prove that the proliferating, quescient and dead cells are vanishing in the healthy tissue we need to prove the following lemma.
\begin{lemma}
\label{FL}
Let $Z\in L^{\infty}(0,T;L^{2}(B))$, $Z\geq 0$, $\vc{v}\in W^{1,2}_{0}(B)$ satisfying the following equation
\begin{equation} 
\begin{split}
 \int_{B}  \big( Z \varphi(\tau,\cdot)&- Z_0 \varphi(0,\cdot)\big)dx\\
 & = \int_0^{\tau}\! \int_{B} \left( Z \partial_t \varphi + Z \vc{v} \cdot \Grad_x \varphi + (\vc{G}+\vc{G_Z} ) \varphi\right) dx dt,
\end{split}
\label{Zeq}
\end{equation}
for any $\tau\in[0,T]$ and any test function $\varphi\in C^{1}_{c}([0,T]\times \R^{3})$  and $\vc{G_Z}$ a linear function of $Z$, while $\vc{G}\in L^{\infty}([0,T]\times B)$,  $\vc{G}(\tau,\cdot)\big|_{B\backslash\Omega_{\tau}}=0$. Moreover assume that
\begin{equation}
(\vc{v}-\vc{V})(\tau, \cdot)\cdot\vc{n}\big|_{\Gamma_{\tau}}=0\qquad \text{a.e. $\tau\in (0,T)$}
\label{ueq}
\end{equation}
and that
$$Z_{0}\in L^{2}(\R^{3}), \qquad Z_{0}\geq 0\qquad Z_{0}\big|_{B\backslash\Omega_{0}}=0.$$
Then
\begin{equation}
Z(\tau, \cdot)\big|_{B\backslash\Omega_{\tau}}=0 \qquad \text{for any $\tau\in [0,T]$}.
\label{zeroZ}
\end{equation}
\end{lemma}

\begin{proof}
 In the  proof  it is crucial the  construction of an appropriate test function to be used in the weak formulation of  \eqref{Zeq}. 
For given $\eta>0$ we use  
\begin{equation}
\varphi=\left[\min\left\{\frac{1}{\eta}\Phi;1\right\}\right]^{+}
\label{specialtest}
\end{equation}
and  we  obtain
\begin{equation} 
\begin{split}
 \int_{B\setminus \Omega_{\tau}} \!\!Z \varphi \, dx = &\frac{1}{\eta}\int_0^{\tau}\! \int_{\{0\leq \Phi(t,x)\leq \eta\}} \left( Z \partial_t \Phi + Z \vc{v} \cdot \Grad_x \Phi + (\vc{G}+\vc{G_Z} ) \Phi\right) dx dt\\
&+ \int_0^{\tau}\! \int_{\{\Phi(t,x)> \eta\}} \vc{G_Z} dxdt.
\end{split}
\label{Zeq2}
\end{equation}
We have that
$$Z \partial_t \Phi + Z \vc{v} \cdot \Grad_x \Phi=Z(\partial_t \Phi +\vc{v} \cdot \Grad_x \Phi)=Z(\vc{v}-\vc{V}) \cdot \Grad_x \Phi$$
where by \eqref{d1} and \eqref{ueq}   we get
\begin{equation}
(\vc{v}-\vc{V}) \cdot \Grad_x \Phi\in W^{1,2}_{0}(B\backslash\Omega_{\tau}) \quad \text {for a.e. $t\in (0,\tau)$.}
\label{d2}
\end{equation}
We introduce now the following distance function
\begin{equation}
\delta(t,x)=dist_{\R^{3}}[x, \partial(B\backslash\Omega_{\tau})] \qquad \text{for $t\in [0,\tau]$, $x\in B\backslash\Omega_{\tau}$}.
\label{dist}
\end{equation}
From \eqref{d2} and an application of Hardy's inequality (see Theorem 21.5 in \cite{OpFed90}) it follows  that
\begin{equation}
\frac{1}{\delta}(\vc{V}-\vc{v}) \cdot \Grad_x \Phi \in L^{2}([0,\tau]\times B\backslash\Omega_{\tau} ).
\label{Hardy}
\end{equation}
On the other hand by taking into account that  $\vc{G}$ is bounded and that $\vc{G_{Z}}$ is a linear function of $Z$ and  $Z\in L^{\infty}(0,T;L^{2}(B))$ and that \eqref{dist} is defined in $\mathbb{R}^{3}$ we have also 
\begin{equation}
\frac{\vc{V}+\vc{G_{Z}}}{\sqrt{\delta}}\in L^{1}([0,\tau]\times B\backslash\Omega_{\tau} ).
\label{Hardy2}
\end{equation}
Since $\vc{V}$ is regular we have that 
\begin{equation}
\frac{\delta(t,x)}{\eta}\leq c,\qquad  \frac{\sqrt{\delta(t,x)}}{\eta}\leq c
\qquad\text{when $0\leq \Phi(t,x)\leq \eta$}.
\label{dist2}
\end{equation}
Going back to \eqref{Zeq2} we get
\begin{equation}
\begin{split}
& \int_{B\backslash\Omega_{\tau}} Z\varphi dx\leq \frac{1}{\eta}\int_0^{\tau}\! \int_{\{0\leq \Phi(t,x)\leq \eta\}} \delta \frac{Z(\vc{v}-\vc{V}) \cdot \Grad_x \Phi}{\delta}dx dt\\ \\
&+\frac{1}{\eta}\int_0^{\tau}\! \!\int_{\{0\leq \Phi(t,x)\leq \eta\}}
\sqrt{\delta}\frac{\vc{G}+\vc{G_{Z}}}{\sqrt{\delta}}\Phi dxdt,
+\int_0^{\tau}\! \!\int_{B\backslash\Omega_{t}} \vc{G_Z} dxdt
\end{split}
\label{Zeq3}
\end{equation}
and letting $\eta\to 0$ in \eqref{Zeq3} and by taking into account \eqref{specialtest}, \eqref{Hardy}, \eqref{Hardy2},  and  that $\vc{G_{Z}}$ is a linear function of $Z$ and  $Z\in L^{\infty}(0,T;L^{2}(B))$, by applying Gronwall's inequality we conclude with
$$ \int_{B\backslash\Omega_{\tau}} Zdx=0.$$
Therefore by using the fact that $Z\geq 0$  and $Z\in L^{\infty}(0,T;L^{2}(B))$ we end up with \eqref{zeroZ}.
\end{proof}
By means of the previous lemma we are able to prove now that the proliferating, quiescent, dead cells a are vanishing in the healthy tissue. 
\begin{proposition}
Assume that $P_{\omega}$, $Q_{\omega}$, $D_{\omega}$ and $C_{\omega}$ satisfy
\text{\bf (IIp)} and that {\text {\bf (IC-p)}} holds, then
\begin{equation}
 P_{\omega}(\tau, \cdot)|_{B\backslash\Omega_{\tau}}=0, \quad  Q_{\omega}(\tau, \cdot)|_{B\backslash\Omega_{\tau}}=0, \quad  D_{\omega}(\tau, \cdot)|_{B\backslash\Omega_{\tau}}=0.
 \label{sdPQD}
\end{equation}
\label{fondprop}
\end{proposition}
\begin{proof}
We start the proof with $P_{\omega}$. Byusing the Proposition \ref{sdCW} and the uniform bounds of the Section \ref{S5} we see  $P_{\omega}$ verifies the hypotheses of the Lemma \ref{FL} if we take 
$$\vc{G}=K_{p}C_{\omega}Q_{\omega}$$ and 
$$\vc{G_{z}}=\left(K_B C_{\omega} - K_Q (\bar C- C_{\omega}) - K_A(\bar C - C_{\omega})\right) P_{\omega}- i_1 G_1(W_{\omega}) P_{\omega},$$ 
so we have that  $P_{\omega}(\tau, \cdot)|_{B\backslash\Omega_{\tau}}=0$. Having obtained the result for $P_{\omega}$, the remaing part of the proof follows with the same type of arguments applied to $Q_{\omega}$ and $D_{\omega}$.
\end{proof}
Now, taking into account the Proposition \ref{fondprop}, $P_{\omega}$, $Q_{\omega}$, $D_{\omega}$ satisfy the weak formulation {\bf (I)}  as $\e\to 0$.
\subsection{Vanishing viscosity limit $\omega \to 0$}
\label{Sv}
The last step in the proof is to perform the limit $\omega \to 0$ in order to get rid of the last viscosity terms of \eqref{pmom} in $B\backslash\Omega_{t}$. By using \eqref{bv} we have that
\begin{equation}
\begin{split}
\int_{\Omega_{t}}\mu\left(|\nabla_{x}\vc{v}_{\omega}|^{2}+|\vc{v}_{\omega}|^{2}\right)dx\leq c\\
\int_{B\backslash\Omega_{t}}\mu_{\omega}\left(|\nabla_{x}\vc{v}_{\omega}|^{2}+|\vc{v}_{\omega}|^{2}\right)dx\leq c,
\end{split}
\label{visc2}
\end{equation}
The estimates  \eqref{visc2} with a standard computations yields that
\begin{equation}
\int_{B\backslash\Omega_{t}}\mu_{\omega}\left(\Grad_x \vc{v}_{\omega}  : \Grad_x \vc{\varphi}  +\vc{v}_{\omega,\e} \vc{\varphi}\right)dx\to 0 
\quad \text{as $\omega\to 0$}.
\label{visc5}
\end{equation}
By combining   \eqref{pmom} with \eqref{visc5}    we get 
$$\int_{B\backslash\Omega_{t}}\sigma_{\omega}\Div \vc{\varphi} dxdt=0,$$
for any text function $\varphi$.
Now in the same spirit of \cite{Donatelli-Trivisa-2008} we can let $\omega\to 0$ in the weak formulations   \eqref{pmom}, \eqref{p-C2}, \eqref{p-Ww} and we complete the proof of Theorem  \ref{T2.2}.

\section{Acknowlegments}
The work of D.D. was supported  by the Ministry of Education, University and Research (MIUR), Italy under the grant PRIN 2012- Project N. 2012L5WXHJ, \emph{Nonlinear Hyperbolic Partial Differential Equations, Dispersive and Transport Equations: theoretical and applicative aspects}. Ê K.T.  gratefully acknowledges the support  in part by the National Science Foundation under the grant DMS-1211519 and by the Simons Foundation under the Simons Fellows in Mathematics Award 267399. 
Part of this research was performed during the visit of K.T. at University of L'Aquila which was supported under the grant PRIN 2012- Project N. 2012L5WXHJ, \emph{Nonlinear Hyperbolic Partial Differential Equations, Dispersive and Transport Equations: theoretical and applicative aspects}.
This work was completed while K.T. was resident at \'{E}cole Normale Sup\'{e}rieure de  Cachan as a Simons Fellow. K.T. is grateful to C.\ Bardos, L.\ Desvilettes and the CMLA Lab for providing a very stimulating environment for scientific research  and to C.\ Villani and  the Institute Henri Poincar\'{e} for the hospitality.


\begin{thebibliography}{10}

\bibitem{AS67}
D. G.\ Aronson, and J.~Serrin, 
  Local behavior of solutions of quasilinear parabolic equations, {\em Arch. Rational Mech. Anal.},
{\bf 25}, (1967), 81--122.
     
\bibitem{CareyKrishnan-1982} G.\ Carey and R.\ Krishnan, Penalty approximation of Stokes flow, Parts I \& II, {\em Comput. Methods Appl. Mech. Engrg.} {\bf 35}, (1982), 169-206.

\bibitem{CareyKrishnan-1984} G.\ Carey and R.\ Krishnan, Penalty finite element method for the Navier-Stokes equations, Parts I \& II, {\em Comput. Methods Appl. Mech. Engrg.} {\bf 42}, (1984) 183-224.

\bibitem{CareyKrishnan-1985} G.\ Carey and R.\ Krishnan,  Continuation techniques for a penalty approximation of the Navier-Stokes equations, {\em Comput. Methods Appl. Mech. Engrg.} {\bf 48}, (1985) 265-282.

\bibitem{ChenFriedman-2013}  D.\ Chen and A.\ Friedman, A two-phase free boundary problem with discontinuous velocity: Applications to tumor model, {\em J. Math. Anal. Appl.} {\bf 399}, (2013), 378-393.

\bibitem{Courant-1956} R.\ Courant, {\em Calculus of Variation and Supplementary Notes and Exercises}, New York University, New York, NY, 1956.


\bibitem{Donatelli-Trivisa-2006}
D.\ Donatelli, K.\ Trivisa, On the motion of a viscous compressible radiative-reacting gas.{\em  Comm. in Math. Phys.}, {\bf 265}, (2006), no. 2, 463-491.

\bibitem{Donatelli-Trivisa-2008}
D.\ Donatelli, K.\ Trivisa, From the dynamics of gaseous stars to the incompressible Euler equations. {\em J. Differential Equations}, {\bf 245}, (2008) 1356-1385-606.




\bibitem{DT-MixedModel-2013} D.\ Donatelli, K.\ Trivisa, On a nonlinear model for tumor growth: Global in time weak solutions, {\em J. Math. Fluid Mech.}, {\bf 16}, (2014), 787--803.

%
\bibitem{FeireislNS-2011} 
E.\ Feiresl, J.\ Neustupa, J.\ Stebel, Convergence of a Brinkman-type penalization for compressible fluid flows. {\em J. Differential Equations}, {\bf 250},  no.1, (2011) 596-606.

\bibitem{FeireislKNNS-2013} 
E.\ Feireisl, O. \ Kreml, S.\ Necasova, J. Neustupa, J. Stebel, Weak solutions to the barotropic Navier-Stokes system with slip boundary conditions in time dependent domains, {\em J.\ Differential Equations} {\bf 254}, (2013) 125-140. 

\bibitem{Friedman-2004}	
A.\ Friedman, A hierarchy of cancer models and their mathematical challenges, {\em Discrete and Continuous Dynamical Systems}, {\bf 4},  no. 1 (2004) 147-159.

\bibitem{ChenFriedman-2013} 
D.\ Chen and A.\ Friedman, A two-phase free boundary problem with discontinuous velocity: Applications to tumor model, {\em J. Math. Anal. Appl.} {\bf 399} (2013) 378-393.


\bibitem{Lions-1998} 
P.-L. Lions, {\em Mathematical topics in Fluid Dynamics,} {\bf Vol. 2} {\em Compressible models,} Oxford Science Publication, Oxford, 1998.

\bibitem{Novotny-Stras-2004} A.~Novotny, I.~Straskraba, {Introduction to the Mathematical Theory of Compressible Flow}, Oxford Science Publication, Oxford, 2004.


\bibitem{OpFed90}
B.~Opic, A.~Kufner, {\em Hardy-type inequalities}. Longman, Pitchman Research Notes in Math. {\bf 219}, Essex, 1990.

\bibitem{OshFed03}
S. Osher, R. Fedwik, Level Set Methods and Dynamic Implicit Surfaces, {\em Appl. Math. Sci.}, {\bf 153}, Springer- Verlag, New York, 2003.

\bibitem{Roda-etal-2011} 
J. M. \ Roda, L.A.\ Summer, R.\  Evans, G.S.\  Philips, C.B. Marsh and  T.D. \ Eubank, Hypoxia inducible factor-2 regulates GM-CSF-derived soluble vascular endothelial growth factor receptor 1 production from macrophages and inhibits tumor growth and angiogenesis. {\em J. Immunol.,} {\bf 187}, (2011), 1970--1976.

\bibitem{Roda-etal-2012} 
J. \ Roda, Y.\  Wang, L.\ Sumner, G.\ Phillips, T.\  Eubank, and C. \ Marsh,  Stabilization of HIF-2
induces SVEGFR-1 production from Tumor-associated macrophages and enhances the Anti-tumor
effects of GM-CSF in murine melanoma model. {\em J. Immunol.,} {\bf 189}, (2012), 3168--3177.

%
\bibitem{Sohr-2001}
H.\ Sohr, {\em The Navier Stokes Equations, An Elementary Functional Analytical Approach}, Birkh\'auser Verlag, Basel, 2001.


\bibitem{StokesCarey-2011} 
Y.\ Stokes and G.\ Carey,  On generalized penalty approaches for slip surface and related boundary conditions in viscous flow simulation, {\em Internat. J. Numer. Methods Heat Fluid Flow}, {\bf 21} (2011) 668-702.

\bibitem{WardKing-1997} J.P.\ Ward and J.R.\ King, Mathematical modeling of avascular amour growth I, {\em IMA J.\ Math.\ Appl.\ Med.\ Biol.\ } {\bf 14}, no. 1 (1997), 39--69.

\bibitem{WardKing-1999}  J.P.\ Ward and J.R.\ King, Mathematical modeling of avascular amour growth II: Modelling growth saturation, {\em IMA J.\ Math.\ Appl.\ Med.\ Biol.\ } {\bf 16},  (1999), 171-211.

\bibitem{WardKing-2003}  J.P.\ Ward and J.R.\ King, Mathematical modeling of drug transport model in tumor multicell spheroids, {\em Math. Biosci.}, {\bf 181} (2003), 177-207.

\bibitem{Zhao-2010} 
J.-H.\ Zhao, A parabolic-hyperbolic free boundary problem modeling tumor growth with drug application,  {\em Electronic Journal of Differential Equations,} {\bf 2010}, (2010) 1--18.
\end{thebibliography}
\end{document}